\documentclass[a4paper,11pt, dvipsnames]{article} 
\usepackage{fourier}
\usepackage{faktor}
\usepackage{amssymb}
\usepackage{amsmath}
\usepackage{amsthm}
\usepackage{stmaryrd}
\usepackage[all]{xy}
\usepackage[shortlabels]{enumitem}
\usepackage{xcolor}
\usepackage{tensor} 

\usepackage{graphicx}

\def\Q{\mathbb{Q}}
\def\Z{\mathbb{Z}}
\def\C{\mathbb{C}}

\def\F{\mathbb{F}}

\newcommand{\cB}{{\mathcal B}}
\newcommand{\cE}{{\mathcal E}}
\newcommand{\cF}{{\mathcal F}}
\newcommand{\cG}{{\mathcal G}}

\newcommand{\cM}{{\mathcal M}}
\newcommand{\cO}{{\mathcal O}}
\newcommand{\cS}{{\mathcal S}}
\renewcommand{\mod}{\bmod}

\newcommand{\bcS}{\overline{\mathcal S}}

\newcommand{\p}{{\mathfrak p}}
\newcommand{\frP}{{\mathfrak P}}
\newcommand{\frI}{{\mathfrak I}}
\newcommand{\frA}{{\mathfrak A}}

\newcommand{\set}[1]{\left\lbrace#1\right\rbrace }

\DeclareMathOperator{\Tr}{Tr}
\DeclareMathOperator{\Hom}{Hom}

\DeclareMathOperator{\Aut}{Aut}
\DeclareMathOperator{\ICM}{ICM}
\DeclareMathOperator{\Pic}{Pic}
\DeclareMathOperator{\End}{End}
\DeclareMathOperator{\GL}{GL}
\DeclareMathOperator{\rank}{rank}

\DeclareMathOperator{\Gal}{Gal}
\DeclareMathOperator{\id}{id}
\DeclareMathOperator{\Fr}{Fr}
\DeclareMathOperator{\Tors}{Tors}

\newcommand{\BassMod}[2]{{\cB}({#1,#2})}
\newcommand{\vphi}{{\varphi}}

\DeclareMathOperator{\AV}{AV}

\newcommand{\AVord}[1]{\AV^{\text{ord}}({#1})}

\newcommand{\Modord}[1]{\cM^{\text{ord}}({#1})}

\newcommand{\Mod}[1]{\cM({#1})}

\newcommand{\Ford}{\cF^{\text{ord}}}

\renewcommand{\bar}{\overline}

\newtheorem{df}{Definition}[section]
\newtheorem{prop}[df]{Proposition}

\newtheorem{thm}[df]{Theorem}
\newtheorem{cor}[df]{Corollary}
\newtheorem{remark}[df]{Remark}
\newtheorem{example}[df]{Example}

\title{Computing base extensions of ordinary abelian varieties over finite fields}
\author{Stefano Marseglia}
\date{\vspace{-4ex}}

\AtEndDocument{\bigskip{\footnotesize
  S.~Marseglia, \textsc{Utrecht University} \par
  \textit{E-mail}: \text{s.marseglia@uu.nl}
}}

\begin{document}
\maketitle

\begin{abstract}
We study base field extensions of ordinary abelian varieties defined over finite fields using the module theoretic description introduced by Deligne.
As applications we give algorithms to determine the minimal field of definition of such a variety and to determine whether two such varieties are twists.
\end{abstract}

\section{Introduction}
Abelian varieties over $\C$ can easily be described in terms of tori. Indeed for a complex abelian variety $A$ of dimension $d$ we have an isomorphism $A(\C)\simeq\C^d/L$, where $L$ is a $\Z$-lattice of rank $2d$.
This association does not hold on the whole category of abelian varieties when we move to the wilder realms of positive characteristic.
The reason is that there are objects such as the supersingular elliptic curves with quaternionic endomorphism algebra, which does not admit a $2$-dimensional representation over $\Q$.

Nevertheless, if we restrict our attention to the ordinary abelian varieties (that is, having $p$-torsion of maximal rank) over a finite field $\F_q$ of characteristic $p$, then we can still mimic the result from the complex world.
More precisely, Deligne in \cite{Del69} proved that the category $\AVord{q}$ of ordinary abelian varieties over $\F_q$ is equivalent to the category $\Modord{q}$ of $\Z$-modules with a ``Frobenius endomorphism".
We recall the precise statement of Deligne's theorem in Section \ref{sec:equivalences}.

Fix an isogeny classes of ordinary abelian varieties over $\F_q$, which by Honda-Tate theory 
is uniquely determined by a $q$-Weil polynomial $h$.
The subcategory $\Mod{h}$ of modules in $\Modord{q}$ whose characteristic polynomial of Frobenius is $h$, under some assumptions on $h$, becomes easy to describe in terms of categories of modules and fractional ideals over orders in product of number fields.
Such descriptions are given in \cite{MarAbVar18} and \cite{MarBassPow} and are recalled in Theorem \ref{thm:bassdecom}.

In this paper we use these module-theoretic descriptions and the related computational tools  to study the functor associating to an abelian variety $A$ in $\AVord{q}$ its base extension $A\otimes_{\F_q}\F_{q^r}$ in $\AVord{q^r}$, for a fixed positive integer $r$.
In particular, if $A$ is a simple abelian variety in $\AVord{q}$ with characteristic polynomial of Frobenius $h$, then the characteristic polynomial of $A\otimes_{\F_q}\F_{q^r}$ will be of the form $g^s$ for some irreducible $q^r$-Weil polynomial and positive integer $s$, as we recall in  Section \ref{sec:uptoisogeny}.
See also Remark \ref{rmk:squarefree} for a generalization to non-simple abelian varieties.

Section \ref{sec:uptoiso} contains the main results of the paper, namely Theorem \ref{thm:functors} and Corollary \ref{cor:functors} which describe the functor $- \otimes_{\F_q}\F_{q^r}:\AV(h)\to \AV(g^s)$ in terms of a functor $\cE_2$ defined on the category of modules.
The definition of $\cE_2$ is well suited for computations on the isomorphism level, as we explain in detail in Section \ref{sec:computations}. 
In particular it allows us to explicitly compute twists of abelian varieties and their (minimal) fields of definition.
These two applications are discussed in Sections \ref{sec:twists}, \ref{sec:gal_coh} and \ref{sec:field_def}.
The algorithms developed, which are available on the webpage of the author, allow us to compute explicit examples, which we include in the various sections of the paper.
In particular see Examples \ref{ex:ext}, \ref{ex:Bass}, \ref{ex:notBass}, \ref{ex:coh}, \ref{ex:ec_extension}, \ref{ex:ec_extension2}, \ref{ex:Bass2} and \ref{ex:notBass2}.

It is worth mentioning that there are other descriptions of subcategories of the category of abelian varieties over finite fields in terms of modules with extra structure other than the one of Deligne.
More precisely \cite{CentelegheStix15} deals only with abelian varieties over prime fields $\F_p$, while  in \cite[Appendix]{Lauter02}, \cite{Kani11} and \cite{JKPRSBT18} discuss functors on categories of abelian varieties isogenous to powers of elliptic curves.
Moreover, in \cite{Vogt19} the author studies the behaviour of the functor introduced in \cite{JKPRSBT18} in relation to Galois field extensions.
An equivalence of categories for almost ordinary simple abelian varieties over finite fields of odd characteristic similar to one of Deligne has been described in \cite{OswalShankarEarlyView}.
We chose to work with Deligne’s equivalence because it allows us to deduce results also about powers of abelian varieties of dimension greater than $1$.

\subsection*{Acknowledgements}
I am grateful to Jonas Bergstr\"om and Valentijn Karemaker for their suggestions. 
I thank Christophe Ritzenthaler and Rachel Newton for comments on a preliminary version of the paper, which was included in my PhD thesis written at Stockholm University.

\section{Preliminaries}
\label{sec:equivalences}
Let $\AVord{q}$ be the category of ordinary abelian varieties over $\F_q$.
Consider the category $\Modord{q}$ consisting of pairs $(T,F)$ where $T$ is a finitely generated free $\Z$-module and $F$ is an endomorphism of $T$ such that
\begin{itemize}
   \item $F\otimes \Q$ is semisimple with eigenvalues of complex absolute value $\sqrt{q}$,
   \item the characteristic polynomial $h$ of $F$ is ordinary, that is, exactly half of the roots of $h$ are $p$-adic units, and
   \item there exists an endomorphism $V$ of $T$ such that $FV=q$.
\end{itemize}

\begin{thm}{\cite[Th\'eor\`eme]{Del69}}
\label{thm:deligne}
   There is an equivalence of categories
   \[ \Ford: \AVord{q} \to \Modord{q}. \]
   If $\Ford(A)=(T,F)$ then $\rank_\Z(T)=2\dim(A)$ and $F$ corresponds to the Frobenius endomorphism of $A$.
\end{thm}

Let $A$ be an abelian variety over $\F_q$. 
Denote by $h_A$ the characteristic polynomial of Frobenius of $A$.
Recall that $h_A$ is a $q$-Weil polynomial, that is, a monic polynomial of even degree with integer coefficients  and with complex roots of absolute value equal to $\sqrt{q}$. 
By Honda-Tate theory, the polynomial $h_A$ uniquely determines the isogeny class of $A$, in the sense that, for an abelian variety $B$ over $\F_q$,
\[ A\sim_{\F_q} B \Longleftrightarrow h_A=h_B, \]
where $h_B$ is the characteristic polynomial of the Frobenius of $B$, see \cite{Tate66}.
Moreover, by \cite{Honda68} and \cite{Tate71}, one can use $q$-Weil polynomials to list all isogeny classes of abelian varieties over $\F_q$ of given dimension.

Let $h$ be the characteristic polynomial of Frobenius of an ordinary abelian variety over $\F_q$ and let $\AV(h)$ be the full subcategory of $\AVord{q}$ consisting of abelian varieties in the isogeny class determined by $h$.
Denote by $\cM(h)$ the image of $\AV(h)$ under $\Ford$.

For a squarefree $q$-Weil polynomial $g\in \Z[x]$ we put $K_g=\Q[x]/(g)$ and $\alpha = x\mod g$.
Let $R_g$ be the order $\Z[\alpha,q/\alpha]$ in $K_g$ and denote by $\BassMod{g}{s}$ the category of torsion free $R_g$-modules $M$ of rank $s$, that is, such that $M\otimes K_g\simeq K_g^s$.
In what follows we will always think of such a module as embedded in $K_g^s$.
Note that the objects of the category $\BassMod{g}{1}$ are just fractional $R_g$-ideals.

\begin{thm}{\cite[Theorem 4.1]{MarBassPow}}
\label{thm:powers}
   Assume that $h=g^s$ for some squarefree polynomial $g$. There is an equivalence of categories
   \[ \cG_h:\cM(h) \to \BassMod{g}{s}. \]   
\end{thm}
We briefly describe the functor $\cG_h$.
Given a pair $(T,F)\in \cM(h)$ we have a natural identification between $R_g$ and $\Z[F,V]$ given by $\alpha\mapsto F$, which induces an $R_g$-module structure on $T$.
Denote by $M$ this module and set $\cG_h((T,F)) = M$.

\begin{df}
   Define the functor $\cF_h : \AV(h) \to \BassMod{g}{s}$ as the composition $\cG_h\circ \Ford$.
\end{df}

\section{Isogeny classes and field extensions}
\label{sec:uptoisogeny}
Let $A$ be an abelian variety over $\F_q$ of dimension $g$, let $h=h_A$ be the characteristic polynomial of Frobenius of $A$.
Let $r$ be a positive integer and put $A_r=A \otimes_{\F_q} \F_{q^r}$ and denote by $h_r$ the characteristic polynomial of Frobenius of $A_r$.
Explicitly, if $\dim(A)=d$ and
\[ h = (x-\alpha_1)\cdot \ldots \cdot (x-\alpha_{2d}) \]
over the complex numbers, then 
\[ h_{r} = (x-\alpha^r_1)\cdot \ldots \cdot (x-\alpha^r_{2d}).\]
One has $h_r\in \Z[x]$ and, in particular, $h_r$ is a $q^r$-Weil polynomial of degree $2g$.
All these results are well-known, see for example \cite[Theorem 5.1.15]{Stichtenoth09}.

Recall that an abelian variety $A$ is isotypic if $A$ is isogenous to $B^n$ for some simple abelian variety $B$ and positive integer $n$.

\begin{prop}[{\cite[Prop.~1.2.6.1]{chaiconradoort14}}]
   If $A$ is isotypic then $A_r$ is isotypic for every $r\geq 1$.
\end{prop}
The statement does not hold over an arbitrary field, see {\cite[Ex.~1.2.6]{chaiconradoort14}}
Also, the converse does not hold: if $A_r$ is isotypic then $A$ does not need to be so, as the next example shows.
\begin{example}
   If $A$ is an abelian surface over $\F_{31}$ with characteristic polynomial
   \[ h_A= (x^2 - 3x + 31)(x^2 + 3x + 31) \]
   then $A$ is isogenous to the product of two non-isogenous elliptic curves $E_1$ and $E_2$. On the other hand $E_1$ and $E_2$ become isogenous over $\F_{31^2}$ and indeed the characteristic polynomial of $A_2=A\otimes \F_{31^2}$ is
   \[ h_{A_2}=(x^2 + 53x + 961)^2.\]
\end{example}

An isotypic abelian variety $A$ over $\F_q$ has characteristic polynomial of Frobenius of the form $h=g^s$, where $g$ is an irreducible $q$-Weil polynomial.
The next result is well-known and we include a proof for completeness.

\begin{prop}
\label{prop:charpolyext}
   Let $h$ be the characteristic polynomial of Frobenius of an abelian variety over $\F_q$.
   Assume that $h$ is irreducible. Then for every $r>0$ we have
   $h_{r}=g^s$
   for some irreducible polynomial $g$ and some positive integer $s$, both depending on~$r$.  
   Moreover, if $s=1$ then $A_r$ is simple (over $\F_{q^r}$).
   If $A$ is ordinary, then $s=1$ if and only if $A_r$ is simple .
\end{prop}
\begin{proof}
   Denote by $K_{h}$ the number field $\Q[x]/(h)$ and put $\alpha_h = x \mod h$.
   Consider the ring homomorphism
   \begin{align*}
      \chi:\Q[x] & \longrightarrow K_{h}\\
	      x & \longmapsto \alpha_h^r
   \end{align*}
   and denote by $g$ the generator of the kernel of $\chi$.
   Since $K_h$ is a field, the quotient $K_g=\Q[x]/(g)$ must be a field as well, that is, $g$ is irreducible.
   In particular, $g$ is the minimal polynomial of the roots of $h_{r}$ and hence $h_{r}$ is a power of $g$.
   Observe that $K_h$ is a finite extension of $K_g$, say of degree $[K_h:K_g]=s$.
   Since $h$ and $h_r$ have the same degree, we obtain that $h_r=g^s$.
   
   If $s=1$ then $h_r$ is irreducible, which, by Honda-Tate theory, implies that $A_r$ is a simple abelian variety over $\F_{q^r}$.
   The last statement follows from the fact that a simple ordinary abelian variety has an irreducible characteristic polynomial, see \cite[Theorem 3.3]{Howe95}.
\end{proof}

\section{Isomorphism classes and field extensions}
\label{sec:uptoiso}
Let $h$ be the characteristic polynomial of a simple ordinary abelian variety $A$ over $\F_q$ and $h_r$ the characteristic polynomial of $A_r=A\otimes \F_{q^r}$.
By Proposition \ref{prop:charpolyext} we know that $h_r=g^s$ for some irreducible polynomial $g$.
Put $K_g=\Q[x]/(g)$ and $K_h=\Q[x]/(h)$, and denote by $\alpha_g$ and $\alpha_h$ the classes of $x$ modulo $g$ and modulo $h$, respectively.
Define $R_g=\Z[\alpha_g,q^r/\alpha_g]\subset K_g$ and $R_h=\Z[\alpha_h,q/\alpha_h]\subset K_h$.

Since $h$ is irreducible, by Theorem \ref{thm:powers} the abelian varieties isogenous to $A$ correspond via the functor $\cF_h$ to the fractional ideals of the order $R_h$ and the abelian varieties isogenous to $A_r$ functorially correspond to the modules in $\BassMod{g}{s}$.
We want to understand how the functor $-\otimes_{\F_q} \F_{q^r}$ acts on these categories.

As already shown in the proof of Proposition \ref{prop:charpolyext}, the field $K_g$ is naturally a subfield of $K_h$ where the inclusion is given by $\alpha_g\mapsto \alpha_h^r$.
Equivalently, we have that $K_h$ is a field extension of degree 
\[ [K_h:K_g] = \dfrac{\deg(h)}{\deg(g)} = s. \]
So in particular there exists a polynomial $l\in K_g[y]$ of degree $s$ such that
\[ \xymatrix{
K_g \ar@{^{(}->}[dr] \ar@{^{(}->}[r]	& K_h\\
			        & \dfrac{K_g[y]}{(l)}\ar[u]_\vphi^[left]{\simeq}}
			        \]
where the isomorphism $\vphi$ is given by
\begin{align*}
 \bar y\mapsto \alpha_h,\\
 \alpha_g \mapsto \alpha_h^r
\end{align*}
where $\bar y = y \mod l$.
Observe that
\[ \dfrac{K_g[y]}{(l)} = K_g\oplus \bar y K_g \oplus \ldots \oplus \bar y ^{s-1} K_g \]
and that there is a natural isomorphism of $R_g$-modules
\begin{align*}
   \psi:K_g\oplus \bar y K_g \oplus \ldots \oplus \bar y ^{s-1} K_g & \overset{\sim}{\longrightarrow} \overbrace{K_g\times \ldots \times K_g}^{s-\text{times}} \\
   \sum_{i=0}^{s-1} b_i\bar y^i & \longmapsto (b_0,\ldots,b_{s-1}).
\end{align*}
Consider the functors
\begin{align*}
  \cE_1:\Mod{h} \to \Mod{h_r}\\
  (T,F) \mapsto (T,F^r)
\end{align*}
and
\begin{align*}
  \cE_2:\BassMod{h}{1} \to \BassMod{g}{s}\\
  I \mapsto \psi(\vphi^{-1}(I)),
\end{align*}
the action on morphisms being the obvious one.
Let $\cG_h:\Mod{h}\to \BassMod{h}{1}$ and $\cG_{h_r}:\Mod{h_r}\to \BassMod{g}{s}$ be defined as in Theorem \ref{thm:powers}.

\begin{thm}
\label{thm:functors}
  We have a commutative diagram of functors:
  \[ \xymatrix{	\AV(h)\ar[d]_{-\otimes \F_{q^r}} \ar[r]_\Ford \ar@/^1pc/[rr]^{\cF_h} 	&\Mod{h} \ar[d]_{\cE_1} \ar[r]_{\cG_h}	& \BassMod{h}{1}\ar[d]_{\cE_2}\\
		\AV(h_r)\ar[r]^\Ford \ar@/_1pc/[rr]_{\cF_{h_r}}		&\Mod{h_r} \ar[r]^{\cG_{h_r}}  	&\BassMod{g}{s}} \]
\end{thm}
\begin{proof}
Let $A$ be an abelian variety in $\AV(h)$ and put $\Ford(A)=(T,F)$.
As usual denote $A\otimes \F_{q^r}$ by $A_r$.
Then $\Ford(A_r)=(T,F^r)$ which proves the commutativity of the left square of the diagram.
The commutativity of the right square follows from the above discussion.
\end{proof}
A straightforward generalization of the previous result leads to the following corollary.
\begin{cor}
\label{cor:functors}
  Let $t$ be a positive integer.
  We have a commutative diagram of functors:
  \[ \xymatrix{	\AV(h^t)\ar[d]_{-\otimes \F_{q^r}} \ar[r]_\Ford \ar@/^1pc/[rr]^{\cF_{h^t}} 	&\Mod{h^t} \ar[d]_{\cE_1} \ar[r]_{\cG_{h^t}}	& \BassMod{h}{t}\ar[d]_{\cE_2}\\
		\AV(h^t_r)\ar[r]^\Ford \ar@/_1pc/[rr]_{\cF_{h^t_r}}		&\Mod{h^t_r} \ar[r]^{\cG_{h^t_r}}  	&\BassMod{g}{ts}} \]
\end{cor}

\begin{remark}
\label{rmk:squarefree}
In this section, and in the following ones, we assume that the $q$-Weil polynomial $h$ is irreducible, that is, that $\AV(h)$ is a simple isogeny class.
If we relax this assumption and instead assume that $h$ is squarefree then the extension $\AV(h_r)$ might fail to be a ``pure power''.
More precisely, if $A\in \AV(h)$ has isogeny decomposition
\[ A \sim_{\F_q} B_1\times \ldots \times B_t, \]
where $B_1,\ldots,B_t$ are simple abelian varieties over $\F_q$, then 
\begin{equation}
\label{eq:isogdec}\tag{*}
A_r \sim_{\F_{q^r}} C_1^{s_1}\times \ldots \times C_{t'}^{s_{t'}},
\end{equation}
where $C_1,\ldots,C_{t'}$ are simple abelian varieties over $\F_{q^r}$ and $s_1,\ldots,s_{t'}$ are positive integers, not necessarily equal.
Nevertheless, we can still apply all the results developed in this section and in the following ones if we assume that $h_r=g^s$ for a squarefree ${q^r}$-Weil polynomial, that is, in \eqref{eq:isogdec} we have $s_1=\ldots=s_{t'}=s$.
\end{remark}

\section{Computations in $\BassMod{g}{s}$}
\label{sec:computations}
Let the notation be as in Section \ref{sec:uptoiso}.
In this section, firstly, we describe how to compute representatives of the isomorphism classes of $\BassMod{g}{s}$ and the functor $\cE_2$ in the cases when $h_r$ is irreducible or the order $R_g$ is Bass, see Subsections \ref{subsec:isom_irred_case} and \ref{subsec:isom_bass_case}, respectively.
In practice, it turns out that most isogeny classes $\AV(h_r)$ fall into one of these two cases. 
 
Secondly, in Subsection \ref{subsec:isom_general}, we focus on the problem of determining when two modules in $\BassMod{g}{s}$ are isomorphic.
We present efficient solutions to the problem in the cases described in Subsections \ref{subsec:isom_irred_case} and \ref{subsec:isom_bass_case}.
Morover we describe a general method to solve the isomorphism problem for modules in $\BassMod{g}{s}$, which in practice turns out to be slower than the previous two, and can not be used to list the representatives of the isomorphism classes.

Finally, in Subsection \ref{subsec:appl_abvar} we give examples of computations of base field extension of abelian varieties.

\subsection{Isomorphism classes when $h_r$ is irreducible}
\label{subsec:isom_irred_case}
Assume that the polynomial $h_r$ is irreducible, that is, $h_r=g$ or equivalently $s=1$.
We fix once and for all an isomorphism $K_h\simeq K_g$.
This allows us to identify $R_g$ with a finite index order contained in $R_h$ and consequently we can identify the objects of the category $\BassMod{g}{1}$ with fractional $R_g$-ideals.
Moreover the operation of ideal multiplication induces the structure of a commutative monoid on $\BassMod{g}{1}$.
The set of ideal classes inherits such a structure: we call it the \emph{ideal class monoid} of $R_g$ and denote it by $\ICM(R_g)$.
In \cite{MarICM18} we give an effective algorithm to compute $\ICM(R_g)$.
Moreover, it is easy to determine if two fractional $R_g$-ideals $I_1$ and $I_2$ define the same class in $\ICM(R_g)$.
Recall the definition of \emph{colon ideal}
\[ (I_1:I_2)=\set{a\in K_g : a I_2\subseteq I_1}. \]
\begin{thm}{\cite[Corollary 4.5]{MarICM18}}
    The fractional $R_g$-ideals $I_1$ and $I_2$ are isomorphic if and only if the following two conditions hold:
    \begin{enumerate}[(1)]
        \item \label{item:1} $1\in (I_1:I_2)(I_2:I_1)$;
        \item \label{item:2} the fractional ideal $(I_1:I_2)$ is a principal $(I_1:I_1)$-ideal.
    \end{enumerate}
\end{thm}
If \ref{item:1} is satisfied then $(I_1:I_2)$ is an invertible fractional $(I_1:I_1)$-ideal, and there are well known algorithms to check whether it is principal.
See also \cite[Algorithm 5]{MarICM18}.

\subsection{Isomorphism classes when $R_g$ is Bass}
\label{subsec:isom_bass_case}
	Recall that an order $S$ in a number field $K$ is \emph{Bass} if every overorder is Gorenstein, or equivalently, if the maximal order $\cO_K$ has cyclic index in $S$, that is, the finite $S$-module $\cO_K/S$ is cyclic.
	For the proofs of these statements and other equivalent definitions, see for instance \cite[Theorem 2.1]{LevyWiegand85}.
	Examples of Bass orders are maximal orders and orders in quadratic number fields.
	
	If the order $R_g$ is Bass, the modules in $\BassMod{g}{s}$ can be written up to $R_g$-linear isomorphism as a direct sum of fractional $R_g$-ideals.
	More precisely, we have the following theorem, see \cite{basshy63}
	 or \cite{borevshaf66}.
\begin{thm}
\label{thm:bassdecom}
   Assume that $R_g$ is a Bass order and let $M$ be in $\BassMod{g}{s}$.
   Then there are fractional $R_g$-ideals $J_1,\ldots,J_s$ satisfying
   \[ (J_1:J_1)\subseteq \ldots \subseteq (J_s:J_s) \]
   and elements $v_1,\ldots,v_s$ in $M$ such that
   \[ M = J_1v_1 \oplus \ldots \oplus J_sv_s. \]   
   Moreover, given another module $N$ in $\BassMod{g}{s}$ with decomposition
   \[ N = I_1u_1 \oplus \ldots \oplus I_su_s, \]
   we have that $M$ and $N$ are $R_g$-isomorphic if and only if
   \[ (J_k:J_k)=(I_k:I_k) \]
   for each $k$ and 
   \[ J_1\cdot\ldots\cdot J_s \simeq I_1\cdot\ldots\cdot I_s. \]
\end{thm}

	Let $M$ be in $\BassMod{g}{s}$, with $R_g$ Bass.
	We can explicitly compute a decomposition
	\[ \quad M = J_1v_1 \oplus \ldots \oplus J_sv_s \]
by following the proof of \cite[Lemma 7]{borevshaf66}.
	For completeness we briefly recall the method.
	We start with a $\Z$-basis of $M$
	\[ M=a_1\Z\oplus \ldots \oplus a_l\Z, \]
	where $l=\deg h_r = \dim_\Q K_g^s$.
	Let $S$ be the multiplicator ring of $M$ in $K_g$, that is,
	\[ S=\set{ a\in K_g : aM\subseteq M }. \]
	By \cite[Lemma 6]{BF65} there exists $\phi \in \Hom_\Z(M,\Z)$ such that $a\phi \not\in \Hom_\Z(M,\Z)$ for every $a\in \cO_{K_g}\setminus S$. 
	Let $e_1,\ldots,e_s$ be the orthogonal idempotents of $K_g^s$ and define $v_i$ in $K_g^s$ to be the dual basis with respect to $\phi$, that is, $(e_i\phi)(v_j)=1$ if $i=j$ and $0$ otherwise. 
	Put $v=v_1+\ldots+v_s$.
	Now, each element $w$ of $M$ in $K_g^s$ can be written in a unique way as
	\[ w=\xi v + y \]
	for $\xi \in K_g$ and $y$ orthogonal to $v$.
	Let $I$ be the subset of coefficients $\xi$ in $K_g$ when $w$ runs over all elements of $M$.
	Observe that $I$ is a fractional $R_g$-ideal and one can prove that the multiplicator ring of $I$ is $S$.
	We then have a decomposition
	\[ M=Iv\oplus M_1, \]
	where $M_1$ is an object of $\BassMod{g}{s-1}$ with multiplicator ring containing $S$.
	We can then proceed recursively to obtain the whole decomposition of $M$.
	
\subsection{Isomorphism testing}
\label{subsec:isom_general}
	Let $M_1$ and $M_2$ be two modules in $\BassMod{g}{s}$.	
	If $s=1$ we can use \cite[Algorithm 5]{MarICM18}  to check if they are isomorphic.	
	If $R_g$ is Bass, for any $s\geq 1$, we can use the algorithm described in Subsection \ref{subsec:isom_bass_case}  to compute the decompositions of $M_1$ and $M_2$ and then conclude by using Theorem \ref{thm:bassdecom} together with \cite[Algorithm 5]{MarICM18}.
	
	If $s>1$ and $R_g$ is not Bass we can use the following method.
	Observe that $M_1$ and $M_2$ are isomorphic as $R_g$-modules if and only if they are so as $\Z[\alpha_g]$-modules, since $\Z[\alpha_g]$ has finite index in $R_g$.
	Let $m_1$ (respectively $m_2$) be the matrix that represents multiplication by $\alpha_g$ with respect to any $\Z$-basis of $M_1$ (respectively $M_2$).
	Observe that $m_1$ and $m_2$ are $N\times N$ matrices with integer entries and the same characteristic polynomial $g^s$ and minimal polynomial $g$, where $N=s\deg(g)$.
\begin{thm}
	The modules $M_1$ and $M_2$ are isomorphic in $\BassMod{g}{s}$ if and only if $m_1$ and $m_2$ are conjugate over the integers, that is, if there exists a matrix $U$ in $\GL_N(\Z)$ such that 
	\[ m_1=Um_2U^{-1}. \]
\end{thm}
\begin{proof}
	The theorem is a direct consequence of generalizations of \cite{LaClMD33}.
	Such generalizations can be found in \cite[Theorem 8.1]{MarICM18} and in \cite{HusertPhDThesis17}.
\end{proof}
	The algorithm described in \cite{EickHofmannOBrien19} allows us to test whether $m_1$ and $m_2$ are conjugate over $\Z$.
	Such an algorithm has the advantage of being very general, at the cost of being slower than the methods described above for the particular cases when $s=1$ or $R_g$ is Bass.

\subsection{Applications to abelian varieties}
The algorithms described above allow us to explicitly compute base field extensions of abelian varieties as we show in the next example.
\label{subsec:appl_abvar}
\begin{example}
\label{ex:ext}
   Consider the polynomial $h=x^6 - x^3 + 8$ corresponding to an isogeny class of ordinary abelian $3$-folds over $\F_2$.
   Denote by $K_h$ the number field $\Q[x]/(h)$ and by $\alpha_h$ the class of $x$ in $K_h$.
   It turns out that the order $R_h=\Z[\alpha_h,2/\alpha_h]$ is maximal and has Picard group of order $3$ generated by the prime ideal $\p_h=2R_h+\alpha_hR_h$.
   So in particular there are $3$ isomorphism classes of abelian varieties in this isogeny class, corresponding to
   $R_h$, $\p_h$ and $\p_h^2$.
   
   We now extend the isogeny class to the field $\F_{2^6}$, which means that we look at the polynomial 
   \[h_6= x^6 + 45x^5 + 867x^4 + 9135x^3 + 55488x^2 + 184320x + 262144.\]
   Observe that $h_6=g^3$, where
   \[ g= x^2 + 15x + 64. \]
   Put $K_g=\Q[x]/(g)$, $\alpha_g=x \mod g$ and $R_g=\Z[\alpha_g,64/\alpha_g]$.
   Note that $R_g$ is the maximal order of $K_g$ and that it has a Picard group of order $3$ generated by $\p_g=2R_g+\alpha_gR_g$.
   Since $R_g$ is Bass, using Theorem \ref{thm:bassdecom}, we can see that the isomorphism classes of abelian varieties in the isogeny class determined by $h_6$
   correspond to the direct sums in $\BassMod{g}{3}$
   \begin{align*}
      M_1=R_g\oplus R_g\oplus R_g, \\
      M_2=R_g\oplus R_g\oplus \p_g ,\\
      M_3=R_g\oplus R_g\oplus \p_g^2.
   \end{align*}
   Moreover, using the same notation as in Theorem \ref{thm:powers}, we can verify that
   \begin{align*}
	\cE_2(R_h) = M_1,\\
	\cE_2(\p_h) = M_2,\\
	\cE_2(\p_h^2) = M_3.
   \end{align*}
\end{example}

\section{Twists}
\label{sec:twists}
Recall that two abelian varieties $A$ and $A'$ over $\F_q$ are twists of each other if there exists some $r>1$ such that $A_r\simeq_{\F_{q^r}} A'_r$.
If this is the case we say that $A$ and $A'$ are $r$-twists.
We say that $A'$ is a trivial twist of $A$ if $A\simeq_{\F_q} A'$.

Assume now that $A$ and $A'$ are simple and ordinary.
A necessary condition for $A$ and $A'$ to be $r$-twists is $h_{A_r} = h_{A'_r}$.
For simplicity of exposition we assume that $A$ and $A'$ are isogenous, say both in $\AV(h)$.
See Remark \ref{rmk:non_isogenous} for an explanation about how to adapt the method descried to the general case.

\begin{prop}
\label{prop:twist}
Let $A$ and $A'$ be simple and ordinary  abelian varieties,  both in the isogeny class $\AV(h)$.
Let $r>1$ and write $h_r=g^s$, with $g$ irreducible.
\begin{enumerate}[(1)]
\item The abelian varieties \label{prop:twist:pow} $A_r$ and $A'_r$ are isomorphic if and only if $\cF_{h_r}(A_r)$ and $\cF_{h_r}(A'_r)$ are isomorphic in $\BassMod{g}{s}$.
\item \label{prop:twist:irr} Moreover, if $s=1$, that is, $h_r$ is irreducible then $A_r$ and $A'_r$ are isomorphic if and only if $A$ and $A'$ are.
\end{enumerate}
\end{prop}
\begin{proof}
Part \ref{prop:twist:pow} follows from Theorem \ref{thm:functors}.
To prove \ref{prop:twist:irr} observe that if $h_r=g$ then
the inclusion $K_g\hookrightarrow K_h$ given by $\alpha_g\mapsto \alpha_h^r$ (discussed in Section \ref{sec:uptoiso}) is an isomorphism.
This implies that the functor $\cE_2:\BassMod{h}{1}\to \BassMod{h_r}{1}$ from Theorem \ref{thm:functors} is fully faithful and hence
\[ \Hom_{\F_q}(A,A')\simeq\Hom_{\F_{q^r}}(A_r,A'_r). \]
\end{proof}

We can use the results contained in Subsection \ref{subsec:isom_general} in order to test the isomorphisms of Proposition \ref{prop:twist}, that is, we have an algorithm to test whether two abelian varieties given as modules in the appropriate category $\BassMod{h}{1}$ are $r$-twists, for a fixed $r>0$.
The implementation of such an algorithm is available on the author's webpage.

\begin{remark}
\label{rmk:non_isogenous}
The assumption that $A$ and $A'$ are isogenous is made to simplify the exposition.
If $A$ and $A'$ are $r$-twists but not necessarily isogenous, say $A\in \AV(h)$ and $A'\in \AV(h')$ with $h_r=h'_r$, then we can still use the theory developed in the previous sections to explicitly compute the corresponding modules and isomorphisms, but we will have to work with the two functors $\cF_{h_r}:\AV(h_r)\to \BassMod{g}{s}$ and  $\cF_{h'_r}:\AV(h'_r)\to \BassMod{g}{s}$.
The implementation of our algorithms includes this case and it is demonstrated in Example \ref{ex:coh}.
\end{remark}

In the reminder of this section we give examples of concrete computations.

\begin{example}
\label{ex:Bass}
	Let $h=x^4 - 205x^2 + 103^2$ and consider the isogeny class of ordinary simple abelian surfaces $\AV(h)$ over $\F_{103}$.
	In $\AV(h)$ there are $12$ isomorphism classes which are represented by the fractional $R_h$-ideals $$\p_3^i\p_5^j$$ for $i=0,1$ and $j=0,\ldots,5$, where $\p_3$ is the unique prime of $R_h$ above $3$ and $\p_5=(5,1+\alpha_h)$.
	Observe that $h_2=g^2$, where $g=x^2 - 205x + 103^2$.
	Moreover, by looking at the square roots of the roots of $h_2$ one can easily verify that there is no $103$-Weil polynomial other than $h$ whose extension gives $h_2$.
	The order $R_g$ is maximal and has cyclic Picard group of order $6$, generated by the class of $\frP=(5,-1+\alpha_g)$.
	In particular the objects of $\BassMod{g}{2}$ can be represented by
	\[  R_g \oplus \frP^k, \text{ for } k=0,\ldots,5. \]
	Using the methods described in Subsection \ref{subsec:isom_general} we compute
	\[ \cE_2(\p_5^j)\simeq\cE_2(\p_3\p_5^j)\simeq R_g \oplus \frP^j \text{ for } j=0,\ldots,5. \]	
	This tells us that for each $j=0,\ldots,5$ the only non-trivial $2$-twist of the abelian variety corresponding to $\p_5^j$ is the abelian variety corresponding to $\p_3\p_5^j$.
	In other words, if we denote by $A_{i,j}$ the abelian variety such that $\cF_h(A_{i,j})=\p_3^i\p_5^j$ then we have 
	\[ A_{0,j}\otimes_{\F_{103}}\F_{103^2}\simeq A_{1,j}\otimes_{\F_{103}}\F_{103^2}. \]
\end{example}

\begin{example}
\label{ex:notBass}
Let $h=x^4 - 18x^2 + 169$ and $g=x^2 - 18x + 169$.
The isogeny class $\AV(h)$ of abelian surfaces over $\F_{13}$ extends to the isogeny class $\AV(g^2)$.
By looking at the complex roots of the polynomial $g$, one can easily check that $h$ is the only $13$-Weil polynomial for which this happens.
The order $R_h$ is not Bass and it has $3$ proper overorders:
\begin{align*}
S & :=R_h+\left( \frac{1+\alpha_h^2}{2}+\frac{\alpha_h^2-5}{26}\alpha_h \right)R_h,\\
T & :=R_h+\frac{\alpha_h^2-5}{26}\alpha_hR_h,
\end{align*}
and the maximal order $\cO_{K_h}$.
Among these orders, the only non-Gorenstein one is $S$.
One computes using the methods described in Subsection \ref{subsec:isom_irred_case} that the $12$ isomorphism classes of $\BassMod{h}{1}$ are represented by
the following set of ideals: 
\[ \set{R_h, I, I^2, I^3, S, IS, S^t, IS^t, T, IT, \cO_{K_h}, I\cO_{K_h}}, \]
where $I$ is a generator of $\Pic(R_h)\simeq \Z/4\Z$ and $S^t$ is the trace dual ideal of $S$, that is, 
\[ S^t=\set{ z \in K_h : \Tr_{K_h/\Q}(zS)\subseteq \Z }. \]
The order $R_g$ is Bass and has only one proper overorder, the maximal order $\cO_{K_g}$.
Observe that by \cite[Corollary 4.3]{MarBassPow} we have that each abelian variety in $\AV(g^2)$ is isomorphic to a product of isogenous elliptic curves.
We have $\Pic(R_g)\simeq\Z/2\Z\times \Z/2\Z$.
We denote by $\frI$ the generator isomorphic to any of the two prime ideals above $47$ and by $\frA$ the other generator.
The isomorphism classes of modules in $\BassMod{g}{s}$ are represented by
\begin{align*}
& M_1=R_g \oplus R_g ,			& & M_2=R_g \oplus \frI,\\
& M_3=R_g \oplus \frA ,		& & M_4=R_g \oplus \frI \frA,\\
& M_5=R_g \oplus \cO_{K_g} ,		& & M_6=R_g \oplus \frI\cO_{K_g}, \\
& M_7=\cO_{K_g} \oplus \cO_{K_g},& & M_8=\cO_{K_g} \oplus \frI\cO_{K_g}.
\end{align*}

We compute that the following isomorphisms of $R_g$-modules hold
\begin{align*}
M_2 \simeq \cE_2(I) \simeq \cE_2(I^3),     		& & M_3 \simeq \cE_2(R_h) \simeq \cE_2(I^2),  \\
M_5 \simeq \cE_2(IS^t) \simeq \cE_2(IS),			& &  M_6 \simeq \cE_2(S^t) \simeq \cE_2(S), \\
M_7 \simeq \cE_2(IT) \simeq \cE_2(\cO_{K_h}),	& &  M_8 \simeq \cE_2(T) \simeq \cE_2(I\cO_{K_h}),
\end{align*}
which allows us to identify the $2$-twists in the isogeny class $\AV(h)$.
\end{example}
\begin{remark}
In Example \ref{ex:notBass} we notice several interesting behaviours.
Since the order $S$ is the unique non-Gorenstein overorder of $R_g$ we deduce that it is must be CM-conjugate stable, that is, $\overline{S}=S$ and in particular, we have that $\overline{S^t}=S^t$ is not isomorphic to $S$.
This tells us that the abelian variety corresponding to $S$ is not isomorphic to its own dual, which corresponds to $\overline{S^t}$ and in particular, it is not principally polarizable.
But the extension $\cE_2(S)\simeq M_6$ corresponds to a product of two elliptic curves which has the product principal polarization. 

Moreover, since $S$ is not Gorenstein, $S^t$ and $S$ are not even weakly equivalent, that is, there exist a prime $\p$ of $S$ such that $S^t_\p\not\simeq S_\p$.
The notion of weak equivalence was introduced in \cite{dadetz62} and in \cite{MarICM18} we give effective algorithms to check whether two fractional ideals are weakly equivalent.
Since $\cE_2(S)\simeq \cE_2(S^t)$ we deduce that the weak equivalence class of an abelian variety does not correspond to a geometric invariant of the corresponding abelian varieties.

Also, the abelian varieties corresponding to $IT$ and $\cO_{K_h}$ which have respectively endomorphism rings isomorphic to $T$ and $\cO_{K_h}$ are $2$-twists.
\end{remark}

\section{Galois cohomology}
\label{sec:gal_coh}
Put $K=\bar{\F}_p$ and $G=\Gal(K/\F_q)$.
Write $\Fr$ for the Frobenius element of $G$.
Let $A$ be an abelian variety over $\F_q$ and put $A_K:=A\otimes K$.
Observe that $\Fr$ acts on $\Aut_K(A)$ by the following rule:
given $\tau \in \Aut_K(A)$ write $\tensor[^{\Fr}]{\tau}{}$ for the twisted automorphism
of $A_K$ defined by
\[ \tensor[^{\Fr}]{\tau}{}= (\id_A \otimes \Fr) \circ \tau \circ (\id_A\otimes \Fr^{-1}). \]
Such an action turns $\Aut_K(A)$ into a topological $G$-module.
Recall that a \emph{cocycle} of $G$ with values in $\Aut_K(A)$ is a $G$-linear map $\varepsilon:G\to \Aut_K(A)$ such that 
\[ \varepsilon( g_1g_2 ) = \varepsilon(g_1)\left( \tensor[^{g_2}]{\varepsilon(g_1)}{} \right), \]
for every $g_1,g_2\in G$.
We denote by $Z^1(G,\Aut_K(A))$ the set of cocycles of $G$ with values in $\Aut_K(A)$.
We say that $\varepsilon_1, \varepsilon_2 \in Z^1(G,\Aut_K(A))$ are \emph{cohomologous} if there exists $ \sigma \in \Aut_K(A) $ such that
\[ \varepsilon_1(g)= \sigma \varepsilon_2(g) \sigma^{-1}, \]
for every $g\in G$.
Observe that being cohomologous defines an equivalence relation on $Z^1(G,\Aut_K(A))$.
The corresponding set of equivalence classes is denoted $ H^1(G,\Aut_K(A)) $.

Denote by $\Theta(A/\F_q)$ the set of 
$\F_q$-isomorphism classes of twists $A'$ (over $\F_q$) of $A$.
The class of $A'$ in $\Theta(A/\F_q)$ is represented by a geometric isomorphism $\phi:A_K\to A'_K$.
Given such $\phi:A_K\to A'_K$ define the map $\varepsilon_\phi:G_{\F_q}\to \Aut_K(A)$ by
\[ \varepsilon_\phi: \alpha \mapsto \phi^{-1}\circ \tensor[^\alpha]{\phi}{}. \]
It is an easy verification that $\varepsilon_\phi \in Z^1(G,\Aut_K(A))$.

Two automorphisms $\tau_1,\tau_2\in \Aut_K(A)$ are called 
\emph{$\Fr$-conjugate}
if there exists $\sigma \in \Aut_K(A)$ such that 
\[ \tau_1= \sigma^{-1}\tau_2(\tensor[^{\Fr}]{\sigma}{}). \]
Being $\Fr$-conjugate is an equivalence relation.
Let $\cS$ be the set automorphisms $\tau \in \Aut_K(A)$ such that there exists $n$ for which 
\begin{enumerate}[(i)]
 	\item \label{def_over_finext} $\tensor[^{\Fr^n}]{\tau}{}=\tau$, and
 	\item $(\tau\cdot\tensor[^{\Fr}]{\tau}{}\cdots \tensor[^{\Fr^{n-1}}]{\tau}{})$ has finite order.
\end{enumerate}
Observe that \ref{def_over_finext} is equivalent to saying that $\tau$ lies in $\Aut_{\F_{q^n}}(A)$ and that the set $\cS$ contains $\Tors(\Aut_K(A))$.
Denote by $\bcS$ the set of $\Fr$-conjugacy classes of elements of $\cS$.

\begin{prop}
\label{prop:coh_bij}
The maps $ \phi \mapsto \varepsilon_\phi$ and $ \varepsilon \mapsto \varepsilon(\Fr_q) $ yield bijections:
\[  \Theta(A,\F_q) \rightarrow H^1(G_{\F_q},\Aut_K(A)) \rightarrow \bcS.\]
\end{prop}
\begin{proof}
By \cite[Ch.III §1.3 Prop.5]{SerreGaloisCohomology} the map $\phi \mapsto \varepsilon_\phi$ induces a bijection
\[ \Theta(A,\F_q) \rightarrow H^1(G_{\F_q},\Aut_K(A)).
\]
Moreover, by \cite[Ch.I §5.1]{SerreGaloisCohomology}, the map $\varepsilon \mapsto \varepsilon(\Fr)$ induces a bijection
\[ H^1(G_{\F_q},\Aut_K(A)) \rightarrow \bcS. \]
\end{proof}
\begin{remark}
Compare Proposition \ref{prop:coh_bij} with \cite[Prop.~3.5]{KaremakerPries19} and \cite[Prop.~5 and Prop.~9]{MeagherTop10}, where there are analogous results for principally polarized abelian varieties and curves over finite fields, respectively.
The main difference with our result is that since we consider unpolarized abelian varieties the automorphism groups are infinite if $\dim(A)>1$.
\end{remark}

\begin{cor}
\label{cor:coh_bij}
Let $A$ be an abelian variety over $\F_q$ such that $\Aut_K(A)=\Aut_{\F_q}(A)$.
Let $\tau \in \Tors(\Aut_{\F_q}(A))$.
Assume that $\tau$ lies in the center of $\Aut_K(A)$ and has order $r$.
Then there exists a twist $\phi:A_r \to A'_r$ such that if we denote by $\pi
$ and $\pi'$ the Frobenius endomorphisms of $A$ and $A'$ respectively, we have
\begin{equation}
\label{cor:coh_bij:eq1}
 \phi^{-1} \circ \pi' \circ \phi = \pi \circ \tau^{-1}.
\end{equation}
In particular $\pi \circ \tau^{-1}$ and $\pi'$ have the same characteristic polynomial.
\end{cor}
\begin{proof}
Since all automorphisms are defined over the base field, $\Fr$-conjugacy coincides with usual conjugacy.
Moreover the conjugacy class of $\tau$ contains only $\tau$.
By the bijections described in Proposition \ref{prop:coh_bij}, the automorphism $\tau$ defines a twist $\phi:A_K \to A'_K$.
Since $\tau$ is defined over $\F_q$ then for every positive integer $n$ we have
\[ \tau \cdot \tensor[^\Fr]{\tau}{} \cdot \tensor[^{\Fr^2}]{\tau}{} \cdots \tensor[^{\Fr^{n-1}}]{\tau}{} = \tau^n \]
and hence by \cite[Rmk.~3.7]{KaremakerPries19} the twist $\phi:A_K\to A'_K$ is defined over $\F_{q^r}$.
Moreover, by \cite[Prop.~3.9]{KaremakerPries19} the twist $\phi:A_r\to A'_r$ satisfies
\[ \phi^{-1} \circ \pi' \circ \phi = \pi \circ \tau^{-1}, \]
as required.
\end{proof}

\begin{cor}
\label{cor:gal_coh_effective}
Let $A$ be a squarefree ordinary abelian variety over $\F_q$.
If the simple isogeny factors of $A$ are absolutely simple, then 
the association
\begin{align*}
\Tors(\Aut_{\F_q}(A)) & \longrightarrow \Theta(A,\F_q)\\
\tau & \longmapsto \phi
\end{align*}
of Corollary \ref{cor:coh_bij} is a bijection.
\end{cor}
\begin{proof}
 By the hypothesis on $A$ we have that $\Aut_{\F_q}(A)$ lies in the center of $\End_K(A)$ and $\End_{\F_q}(A)=\End_K(A)$.
 Moreover the set $\bcS$ equals $\Tors(\Aut_{\F_q}(A))$.
 Proposition \ref{prop:coh_bij} and Corollary \ref{cor:coh_bij} yield the desired bijection.
\end{proof}
 Corollary \ref{cor:gal_coh_effective} allow us to identify which twist is induced by $\tau$ by means of the relation \eqref{cor:coh_bij:eq1}, as we show in Example \ref{ex:coh}.

\begin{example}
\label{ex:coh}
Consider the $5^4$-Weil polynomial 
\[ h_4=x^6 - 112 x^5 + 5872 x^4 - 184786 x^3 + 5872 \cdot 5^4 x^2 - 112\cdot 5^8 x + 5^{12}. \]
The corresponding isogeny class $\AV(h_4)$ can be attained as base field extension of $4$ primitive absolutely simple isogeny classes determined by the following $5$-Weil polynomials
\begin{align*}
h^{(1)}=x^6 + 4x^5 + 12x^4 + & 36x^3 + 60x^2 + 100x + 125,\\ 
h^{(2)}=x^6 - 4x^5 + 12x^4 - & 36x^3 + 60x^2 - 100x + 125, \\
h^{(3)}=x^6 - 4x^4 - & 2x^3 - 20x^2 + 125,\\ 
h^{(4)}=x^6 - 4x^4 + & 2x^3 - 20x^2 + 125.
\end{align*}
The isogeny classes $\AV(h^{(1)}), \AV(h^{(2)}) , \AV(h^{(3)})$ and $\AV(h^{(4)})$ contain $1$, $1$, $14$ and $14$ isomorphism classes of abelian varieties, respectively.
Each isogeny class $\AV(h^{(i)})$ contains a single abelian variety with $4$ distinct torsion automorphisms, which we will denote by $1,-1,\iota_i$ and $-\iota_i$, with orders $1,2,4,4$, respectively.
All the other $26$ isomorphism classes have only two torsion automorphisms, namely $1$ and $-1$.
We do not add a pedix to the automorphisms $1$ and $-1$ since all abelian varieties considered have only one automorphism of order $1$, the identity, and one automorphism or order $2$, so no confusion can arise.
In the following $30\times 30$ matrix in the entry $(i,j)$ we write $"\cdot"$ if the $i$-th  and $j$-th isomorphism classes are not $4$-twists and, otherwise, $1,-1,\iota_i$ or $-\iota_i$ for the torsion automorphism of the $i$-th abelian variety which induces the twist.
\setcounter{MaxMatrixCols}{30}
\begin{center}
\resizebox{\hsize}{!}{
$\left(
\begin{smallmatrix}
1&-1&\cdot&\cdot&\cdot&\cdot&\cdot&\cdot&\cdot&\cdot&\cdot&\cdot&\cdot&\cdot&\cdot&\iota_1&\cdot&\cdot&\cdot&\cdot&\cdot&\cdot&\cdot&\cdot&\cdot&\cdot&\cdot&\cdot&\cdot&-\iota_1\\
-1&1&\cdot&\cdot&\cdot&\cdot&\cdot&\cdot&\cdot&\cdot&\cdot&\cdot&\cdot&\cdot&\cdot&\iota_2&\cdot&\cdot&\cdot&\cdot&\cdot&\cdot&\cdot&\cdot&\cdot&\cdot&\cdot&\cdot&\cdot&-\iota_2\\
\cdot&\cdot&1&\cdot&\cdot&\cdot&\cdot&\cdot&\cdot&\cdot&\cdot&\cdot&\cdot&\cdot&\cdot&\cdot&-1&\cdot&\cdot&\cdot&\cdot&\cdot&\cdot&\cdot&\cdot&\cdot&\cdot&\cdot&\cdot&\cdot\\
\cdot&\cdot&\cdot&1&\cdot&\cdot&\cdot&\cdot&\cdot&\cdot&\cdot&\cdot&\cdot&\cdot&\cdot&\cdot&\cdot&-1&\cdot&\cdot&\cdot&\cdot&\cdot&\cdot&\cdot&\cdot&\cdot&\cdot&\cdot&\cdot\\
\cdot&\cdot&\cdot&\cdot&1&\cdot&\cdot&\cdot&\cdot&\cdot&\cdot&\cdot&\cdot&\cdot&\cdot&\cdot&\cdot&\cdot&-1&\cdot&\cdot&\cdot&\cdot&\cdot&\cdot&\cdot&\cdot&\cdot&\cdot&\cdot\\
\cdot&\cdot&\cdot&\cdot&\cdot&1&\cdot&\cdot&\cdot&\cdot&\cdot&\cdot&\cdot&\cdot&\cdot&\cdot&\cdot&\cdot&\cdot&-1&\cdot&\cdot&\cdot&\cdot&\cdot&\cdot&\cdot&\cdot&\cdot&\cdot\\
\cdot&\cdot&\cdot&\cdot&\cdot&\cdot&1&\cdot&\cdot&\cdot&\cdot&\cdot&\cdot&\cdot&\cdot&\cdot&\cdot&\cdot&\cdot&\cdot&-1&\cdot&\cdot&\cdot&\cdot&\cdot&\cdot&\cdot&\cdot&\cdot\\
\cdot&\cdot&\cdot&\cdot&\cdot&\cdot&\cdot&1&\cdot&\cdot&\cdot&\cdot&\cdot&\cdot&\cdot&\cdot&\cdot&\cdot&\cdot&\cdot&\cdot&-1&\cdot&\cdot&\cdot&\cdot&\cdot&\cdot&\cdot&\cdot\\
\cdot&\cdot&\cdot&\cdot&\cdot&\cdot&\cdot&\cdot&1&\cdot&\cdot&\cdot&\cdot&\cdot&\cdot&\cdot&\cdot&\cdot&\cdot&\cdot&\cdot&\cdot&-1&\cdot&\cdot&\cdot&\cdot&\cdot&\cdot&\cdot\\
\cdot&\cdot&\cdot&\cdot&\cdot&\cdot&\cdot&\cdot&\cdot&1&\cdot&\cdot&\cdot&\cdot&\cdot&\cdot&\cdot&\cdot&\cdot&\cdot&\cdot&\cdot&\cdot&-1&\cdot&\cdot&\cdot&\cdot&\cdot&\cdot\\
\cdot&\cdot&\cdot&\cdot&\cdot&\cdot&\cdot&\cdot&\cdot&\cdot&1&\cdot&\cdot&\cdot&\cdot&\cdot&\cdot&\cdot&\cdot&\cdot&\cdot&\cdot&\cdot&\cdot&-1&\cdot&\cdot&\cdot&\cdot&\cdot\\
\cdot&\cdot&\cdot&\cdot&\cdot&\cdot&\cdot&\cdot&\cdot&\cdot&\cdot&1&\cdot&\cdot&\cdot&\cdot&\cdot&\cdot&\cdot&\cdot&\cdot&\cdot&\cdot&\cdot&\cdot&-1&\cdot&\cdot&\cdot&\cdot\\
\cdot&\cdot&\cdot&\cdot&\cdot&\cdot&\cdot&\cdot&\cdot&\cdot&\cdot&\cdot&1&\cdot&\cdot&\cdot&\cdot&\cdot&\cdot&\cdot&\cdot&\cdot&\cdot&\cdot&\cdot&\cdot&-1&\cdot&\cdot&\cdot\\
\cdot&\cdot&\cdot&\cdot&\cdot&\cdot&\cdot&\cdot&\cdot&\cdot&\cdot&\cdot&\cdot&1&\cdot&\cdot&\cdot&\cdot&\cdot&\cdot&\cdot&\cdot&\cdot&\cdot&\cdot&\cdot&\cdot&-1&\cdot&\cdot\\
\cdot&\cdot&\cdot&\cdot&\cdot&\cdot&\cdot&\cdot&\cdot&\cdot&\cdot&\cdot&\cdot&\cdot&1&\cdot&\cdot&\cdot&\cdot&\cdot&\cdot&\cdot&\cdot&\cdot&\cdot&\cdot&\cdot&\cdot&-1&\cdot\\
\iota_3&-\iota_3&\cdot&\cdot&\cdot&\cdot&\cdot&\cdot&\cdot&\cdot&\cdot&\cdot&\cdot&\cdot&\cdot&1&\cdot&\cdot&\cdot&\cdot&\cdot&\cdot&\cdot&\cdot&\cdot&\cdot&\cdot&\cdot&\cdot&-1\\
\cdot&\cdot&-1&\cdot&\cdot&\cdot&\cdot&\cdot&\cdot&\cdot&\cdot&\cdot&\cdot&\cdot&\cdot&\cdot&1&\cdot&\cdot&\cdot&\cdot&\cdot&\cdot&\cdot&\cdot&\cdot&\cdot&\cdot&\cdot&\cdot\\
\cdot&\cdot&\cdot&-1&\cdot&\cdot&\cdot&\cdot&\cdot&\cdot&\cdot&\cdot&\cdot&\cdot&\cdot&\cdot&\cdot&1&\cdot&\cdot&\cdot&\cdot&\cdot&\cdot&\cdot&\cdot&\cdot&\cdot&\cdot&\cdot\\
\cdot&\cdot&\cdot&\cdot&-1&\cdot&\cdot&\cdot&\cdot&\cdot&\cdot&\cdot&\cdot&\cdot&\cdot&\cdot&\cdot&\cdot&1&\cdot&\cdot&\cdot&\cdot&\cdot&\cdot&\cdot&\cdot&\cdot&\cdot&\cdot\\
\cdot&\cdot&\cdot&\cdot&\cdot&-1&\cdot&\cdot&\cdot&\cdot&\cdot&\cdot&\cdot&\cdot&\cdot&\cdot&\cdot&\cdot&\cdot&1&\cdot&\cdot&\cdot&\cdot&\cdot&\cdot&\cdot&\cdot&\cdot&\cdot\\
\cdot&\cdot&\cdot&\cdot&\cdot&\cdot&-1&\cdot&\cdot&\cdot&\cdot&\cdot&\cdot&\cdot&\cdot&\cdot&\cdot&\cdot&\cdot&\cdot&1&\cdot&\cdot&\cdot&\cdot&\cdot&\cdot&\cdot&\cdot&\cdot\\
\cdot&\cdot&\cdot&\cdot&\cdot&\cdot&\cdot&-1&\cdot&\cdot&\cdot&\cdot&\cdot&\cdot&\cdot&\cdot&\cdot&\cdot&\cdot&\cdot&\cdot&1&\cdot&\cdot&\cdot&\cdot&\cdot&\cdot&\cdot&\cdot\\
\cdot&\cdot&\cdot&\cdot&\cdot&\cdot&\cdot&\cdot&-1&\cdot&\cdot&\cdot&\cdot&\cdot&\cdot&\cdot&\cdot&\cdot&\cdot&\cdot&\cdot&\cdot&1&\cdot&\cdot&\cdot&\cdot&\cdot&\cdot&\cdot\\
\cdot&\cdot&\cdot&\cdot&\cdot&\cdot&\cdot&\cdot&\cdot&-1&\cdot&\cdot&\cdot&\cdot&\cdot&\cdot&\cdot&\cdot&\cdot&\cdot&\cdot&\cdot&\cdot&1&\cdot&\cdot&\cdot&\cdot&\cdot&\cdot\\
\cdot&\cdot&\cdot&\cdot&\cdot&\cdot&\cdot&\cdot&\cdot&\cdot&-1&\cdot&\cdot&\cdot&\cdot&\cdot&\cdot&\cdot&\cdot&\cdot&\cdot&\cdot&\cdot&\cdot&1&\cdot&\cdot&\cdot&\cdot&\cdot\\
\cdot&\cdot&\cdot&\cdot&\cdot&\cdot&\cdot&\cdot&\cdot&\cdot&\cdot&-1&\cdot&\cdot&\cdot&\cdot&\cdot&\cdot&\cdot&\cdot&\cdot&\cdot&\cdot&\cdot&\cdot&1&\cdot&\cdot&\cdot&\cdot\\
\cdot&\cdot&\cdot&\cdot&\cdot&\cdot&\cdot&\cdot&\cdot&\cdot&\cdot&\cdot&-1&\cdot&\cdot&\cdot&\cdot&\cdot&\cdot&\cdot&\cdot&\cdot&\cdot&\cdot&\cdot&\cdot&1&\cdot&\cdot&\cdot\\
\cdot&\cdot&\cdot&\cdot&\cdot&\cdot&\cdot&\cdot&\cdot&\cdot&\cdot&\cdot&\cdot&-1&\cdot&\cdot&\cdot&\cdot&\cdot&\cdot&\cdot&\cdot&\cdot&\cdot&\cdot&\cdot&\cdot&1&\cdot&\cdot\\
\cdot&\cdot&\cdot&\cdot&\cdot&\cdot&\cdot&\cdot&\cdot&\cdot&\cdot&\cdot&\cdot&\cdot&-1&\cdot&\cdot&\cdot&\cdot&\cdot&\cdot&\cdot&\cdot&\cdot&\cdot&\cdot&\cdot&\cdot&1&\cdot\\
-\iota_4&\iota_4&\cdot&\cdot&\cdot&\cdot&\cdot&\cdot&\cdot&\cdot&\cdot&\cdot&\cdot&\cdot&\cdot&-1&\cdot&\cdot&\cdot&\cdot&\cdot&\cdot&\cdot&\cdot&\cdot&\cdot&\cdot&\cdot&\cdot&1
\end{smallmatrix}\right)
$
}
\end{center}
\end{example}
\section{Field of definition}
\label{sec:field_def}
Let $A$ be an abelian variety over $\F_q$ and $k$ be a subfield of $\F_q$.
We say that $A$ is \emph{defined} over $k$ if there exists an abelian variety $A'$ over $k$ such that $A'\otimes_k\F_q$ is isomorphic to $A$ over $\F_q$.
We say that an abelian variety $A$ over $\F_q$ is \emph{primitive} if there is no proper subfield $k$ of $\F_q$ such that $A$ is defined over $k$.
Moreover, for a $q$-Weil polynomial $h$, we say that isogeny class $\AV(h)$ is \emph{primitive} if every abelian variety $A$ in $\AV(h)$ is so. 

Let $h$ be an irreducible ordinary $q$-Weil polynomial.
By looking at the complex roots of $h$ it is easy to list all $q_0$-Weil polynomials $h_0$, with $\F_{q_0}\subseteq \F_q$, that give $h$ after a base field extension to $\F_q$.
Observe that each such $h_0$ is irreducible and for any subfield $\F_{q_0}\subseteq \F_q$ there might be more than one $q_0$-Weil polynomial that extends to $h$.
\begin{example}
\label{ex:ec_extension}
The isogeny class of elliptic curves over $\F_{16}$ determined by $h_{16}=x^2-x+16$ is not primitive and can be attained as a base extension of the primitive isogeny classes $h_4=x^2-3x+4$ over $\F_4$ and of both $h_{2,1}=x^2+x+2$ and $h_{2,2}=x^2-x+2$ over $\F_2$.
\end{example}

\begin{cor}
\label{cor:field_def}
Let $h$ be an irreducible ordinary $q$-Weil polynomial.
Let $t$ be some positive integer.
For $A$ in $\AV(h^t)$, consider the $R_h$-module $M=\cF_{h^t}(A)$ in $\BassMod{h}{t}$.
Let $h_0$ be a $q_0$-Weil polynomial that extends to $h$.
The abelian variety $A$ can be defined over $\F_{q_0}$ if and only if $M$ is an $R_{h_0}$-module, that is, there exists $M_0$ in $\BassMod{h_0}{t}$ such that $\cE_2(M_0)\simeq M$.
\end{cor}
\begin{proof}
This is a consequence of Corollary \ref{cor:functors}.
\end{proof}

\begin{example}
\label{ex:ec_extension2}
Using the same notation as in Example \ref{ex:ec_extension}, we fix isomorphisms between $K_{h_{16}}$, $K_{h_4}$, $K_{h_{2,1}}$ and $K_{h_{2,2}}$ so that we can work in $K_{h_{16}}$, which we will denote by $K$.

The order $R_{h_{16}}$ has index $3$ in the maximal order $\cO_K$.
Moreover the images in $K$ of the orders $R_{h_4}$, $R_{h_{2,1}}$ and $R_{h_{2,2}}$ all equal the maximal order $\cO_K$.
Since the Picard group of $R_{h_{16}}$ has order $4$ and $\cO_K$ is a principal ideal domain, we see that there are $5$ isomorphism classes of elliptic curves in $\AV(h_{16})$.
The first $4$ have endomorphism ring isomorphic to $R_{h_{16}}$ and so by Corollary \ref{cor:field_def} cannot be defined over any  proper subfield of~$\F_{16}$.
On the other hand, the unique isomorphism class with maximal endomorphism ring can be defined over $\F_4$ or over $\F_2$.
It is not hard to determine equations of the representatives of these classes.
Write $\F_{16}=\F_2(T)$ and $\F_4=\F_2(S)$, for $T^4+T+1 = 0$ and $S^2+S+1=0$.
Consider the elliptic curves
\begin{align*}
   & E_{16,i} : y^2 + xy = x^3 + T^{2i} \in \AV(h_{16}), \qquad \text{for }i=0,1,2,3,4,\\
   & E_{2,1} : y^2 + xy = x^3 + 1 \in \AV(h_{2,1}),\\
   & E_{2,2} : y^2 + xy + y = x^3 + 1 \in \AV(h_{2,2}),\\
   & E_4 : y^2 + xy + Sy = x^3 + S \in \AV(h_4).
\end{align*}
We have that $E_{16,0}$ is isomorphic to $E_{2,1}\otimes_{\F_2}\F_{16}$, $E_{2,2}\otimes_{\F_2}\F_{16}$ and $E_4\otimes_{\F_4}\F_{16}$. We deduce that $E_{16,0}$ has maximal endomorphism ring, while $E_{16,i}$ for $i=1,\ldots,4$ represent the isomorphism classes with endomorphism ring isomorphic to $R_{16}$.

\end{example}

\begin{example}
\label{ex:Bass2}
Consider the situation of Example \ref{ex:Bass}.
From the computations described, we see that the $6$ isomorphism classes of abelian varieties in $\AV(h_2)$ are extensions of abelian varieties from $\AV(h)$, that is, they can all be defined over $\F_{103}$.
\end{example}

\begin{example}
\label{ex:notBass2}
Consider Example \ref{ex:notBass}.
Here we see that not all isomorphism classes in $\AV(h_2)$ are extensions.
Indeed the varieties corresponding to the modules $M_1$ and $M_4$ cannot be defined over the prime field $\F_{13}$.
\end{example}

\bibliographystyle{amsalpha}
\newcommand{\etalchar}[1]{$^{#1}$}
\providecommand{\bysame}{\leavevmode\hbox to3em{\hrulefill}\thinspace}

\end{document}